\documentclass{amsart}
\usepackage{lmodern}
\usepackage{microtype}
\usepackage[english]{babel}
\usepackage{hyperref}

\theoremstyle{plain} 
\newtheorem{theorem}{Theorem} 
\newtheorem{lemma}[theorem]{Lemma} 
\newtheorem{corollary}[theorem]{Corollary} 

\theoremstyle{remark} 
\newtheorem*{remark}{Remark}

\begin{document} 
\title{Minimal norm Hankel operators} 
\date{\today} 

\author{Ole Fredrik Brevig} 
\address{Department of Mathematics, University of Oslo, 0851 Oslo, Norway} 
\email{obrevig@math.uio.no} 
\begin{abstract}
	Let $\varphi$ be a function in the Hardy space $H^2(\mathbb{T}^d)$. The associated (small) Hankel operator $\mathbf{H}_\varphi$ is said to have minimal norm if the general lower norm bound $\|\mathbf{H}_\varphi\| \geq \|\varphi\|_{H^2(\mathbb{T}^d)}$ is attained. Minimal norm Hankel operators are natural extremal candidates for the Nehari problem. If $d=1$, then $\mathbf{H}_\varphi$ has minimal norm if and only if $\varphi$ is a constant multiple of an inner function. Constant multiples of inner functions generate minimal norm Hankel operators also when $d\geq2$, but in this case there are other possibilities as well. We investigate two different classes of symbols generating minimal norm Hankel operators and obtain two different refinements of a counter-example due to Ortega-Cerd\`{a} and Seip. 
\end{abstract}


\maketitle

\section{Introduction} Let $\mathbb{T}^d$ denote the $d$-dimensional torus and equip $\mathbb{T}^d$ with its Haar measure. The Hardy space $H^p(\mathbb{T}^d)$ is the subspace of $L^p(\mathbb{T}^d)$ comprised of functions whose Fourier coefficients are supported on $\mathbb{N}_0^d$, where $\mathbb{N}_0=\{0,1,2,\ldots\}$. Let $\overline{H^2}(\mathbb{T}^d)$ be the subspace of $L^2(\mathbb{T}^d)$ comprised of the complex conjugates of functions in $H^2(\mathbb{T}^d)$. The orthogonal projections from $L^2(\mathbb{T}^d)$ to $H^2(\mathbb{T}^d)$ and from $L^2(\mathbb{T}^d)$ to $\overline{H^2}(\mathbb{T}^d)$ will be denoted $P$ and $\overline{P}$, respectively.

For a \emph{symbol} $\varphi$ in $H^2(\mathbb{T}^d$), we consider the associated (small) Hankel operator 
\begin{equation}\label{eq:hankel} 
	\mathbf{H}_\varphi f = \overline{P}(\overline{\varphi} f) 
\end{equation}
which maps $H^2(\mathbb{T}^d)$ to $\overline{H^2}(\mathbb{T}^d)$. The lower and upper norm estimates 
\begin{equation}\label{eq:lu} 
	\|\varphi\|_{H^2(\mathbb{T}^d)} \leq \|\mathbf{H}_\varphi\| \leq \|\varphi\|_{H^\infty(\mathbb{T}^d)} 
\end{equation}
are both well-known and trivial. We say that the Hankel operator $\mathbf{H}_\varphi$ has \emph{minimal norm} if it attains the lower bound in \eqref{eq:lu}.

Recall that a function $I$ in $H^2(\mathbb{T}^d)$ is called \emph{inner} whenever $|I(z)|=1$ for almost every $z$ in $\mathbb{T}^d$. If $\varphi = C I$ for a constant $C$ and an inner function $I$, then clearly 
\begin{equation}\label{eq:inner} 
	\|\varphi\|_{H^2(\mathbb{T}^d)} = \|\varphi\|_{H^\infty(\mathbb{T}^d)} = |C| 
\end{equation}
and consequently $\mathbf{H}_\varphi$ has minimal norm by \eqref{eq:lu}. It turns out that there are no other minimal norm Hankel operators on the one-dimensional torus. 
\begin{theorem}\label{thm:inner1} 
	Suppose that $\varphi$ is in $H^2(\mathbb{T})$. Then $\mathbf{H}_\varphi$ has minimal norm if and only if $\varphi$ is a constant multiple of an inner function. 
\end{theorem}

By orthogonality, the upper bound in \eqref{eq:lu} can be improved to 
\begin{equation}\label{eq:Linf} 
	\|\mathbf{H}_\varphi\| \leq \inf\left\{\|\psi\|_{L^\infty(\mathbb{T}^d)}\,:\, P\psi = \varphi\right\}. 
\end{equation}
If the Hankel operator $\mathbf{H}_\varphi$ is bounded, then the \emph{Nehari problem} is to find a function $\psi$ attaining the infimum on the right hand side of \eqref{eq:Linf}. By \eqref{eq:lu} and \eqref{eq:inner}, it is clear that if $\varphi$ is a constant multiple of an inner function, then a solution to the Nehari problem is trivially $\psi=\varphi$.

Nehari \cite{Nehari57} established that on the one-dimensional torus, the problem always has a solution $\psi$ which satisfies $\|\mathbf{H}_\varphi\|=\|\psi\|_{L^\infty(\mathbb{T})}$. In general, let $C_d\geq1$ denote the smallest real number such that 
\begin{equation}\label{eq:Cd} 
	\inf\left\{\|\psi\|_{L^\infty(\mathbb{T}^d)}\,:\, P\psi = \varphi\right\} \leq C_d \|\mathbf{H}_\varphi\| 
\end{equation}
for every $\varphi$ in $H^2(\mathbb{T}^d)$. The non-trivial part of Nehari's theorem is that $C_1=1$. Ortega-Cerd\`{a} and Seip~\cite{OCS12} found a sequence of polynomials which demonstrates that if $d$ is even, then 
\begin{equation}\label{eq:OCS} 
	C_d \geq \left(\frac{\pi^2}{8}\right)^\frac{d}{4}. 
\end{equation}
The arguments in \cite{OCS12} also imply that every polynomial in the sequence generates a minimal norm Hankel operator. In hindsight, this is perhaps not very surprising. If $\mathbf{H}_\varphi$ has minimal norm, then we in a sense minimize the right hand side of \eqref{eq:Cd}.

The present paper grew out of a desire to put the polynomials from \cite{OCS12} in context. Another source of motivation is the fact that characterizations of inner functions in dimension one, which in our case is provided by Theorem~\ref{thm:inner1}, often lead to a rich theory in higher dimensions. A similar phenomenon can be encountered in the recent paper \cite{BOCS21}. 

We will study two different classes of symbols generating minimal norm Hankel operators, both inspired by $\varphi(z)=z_1+z_2$ which is the basic case in the construction of \cite{OCS12}. In the first class, we think of $\varphi$ as a sum of two inner functions in separate variables. In the second class, we consider $\varphi$ as a $1$-homogeneous polynomial.

Our first main result provides sufficient conditions on when the product or sum of symbols generating minimal Hankel norm operators again will generate minimal norm Hankel operators. 
\begin{theorem}\label{thm:mnorm} 
	Suppose that $\varphi_1$ and $\varphi_2$ in $H^2(\mathbb{T}^d)$ depend on separate variables and that both $\mathbf{H}_{\varphi_1}$ and $\mathbf{H}_{\varphi_2}$ have minimal norm. 
	\begin{enumerate}
		\item[(a)] $\mathbf{H}_{\varphi_1 \varphi_2}$ has minimal norm. 
		\item[(b)] If additionally $\varphi_1(0)=\varphi_2(0)=0$, then $\mathbf{H}_{\varphi_1+\varphi_2}$ has minimal norm. 
	\end{enumerate}
\end{theorem}

Theorem~\ref{thm:mnorm} suggests the following recipe for constructing symbols generating minimal norm Hankel operators. 
\begin{enumerate}
	\item[(i)] Choose any number of (not necessarily distinct) inner functions vanishing at the origin. 
	\item[(ii)] If necessary, rename the variables to ensure that the inner functions depend on mutually separate variables. 
	\item[(iii)] Combine these functions using linear combinations and multiplications, but make sure to use each function only once. 
\end{enumerate}
The polynomials used by Ortega-Cerd\`{a} and Seip fit into this framework as follows. Choose $2d$ copies of the inner function $I(z)=z$ in $H^2(\mathbb{T})$ and rename the variables $z_1,z_2,\ldots,z_{2d}$. Take the pairwise sum of these functions, obtaining $z_1+z_2$, $z_3+z_4$, all the way up to $z_{2d-1}+z_{2d}$. Finally, multiply together these $d$ functions to obtain
\[\varphi_d(z) = \prod_{j=1}^{d} (z_{2j-1}+z_{2j}).\]
In view of Theorem~\ref{thm:mnorm}, we \emph{know} that the resulting Hankel operator $\mathbf{H}_{\varphi_d}$ has minimal norm. Consequently, $\|\mathbf{H}_{\varphi_d}\|=\|\varphi_d\|_{H^2(\mathbb{T}^{2d})}=2^{d/2}$. In \cite{OCS12}, this fact has to be established using the Schur test.

Since $\pi\geq3$, we see from \eqref{eq:OCS} that $C_d \to \infty$ as $d\to\infty$. By a contradiction to the Closed Graph Theorem this demonstrates that there are $\varphi$ in $H^2(\mathbb{T}^\infty)$ such that $\mathbf{H}_\varphi$ is bounded, but for which the corresponding Nehari problem has no solution $\psi$ in $L^\infty(\mathbb{T}^\infty)$. This allowed the authors of \cite{OCS12} to complete a research program initiated by Helson \cite{Helson05,Helson06,Helson10}.

Using the recipe outlined above, we can revisit the counter-example from \cite{OCS12} and exhibit an explicit symbol $\varphi$ in $H^2(\mathbb{T}^\infty)$ for which the Nehari problem has no solution in the following strong sense. 
\begin{theorem}\label{thm:cex} 
	Consider
	\[\varphi(z) = \frac{\sqrt{6}}{\pi}\sum_{k=1}^\infty \frac{1}{k} \prod_{j=(k-1)k/2+1}^{k(k+1)/2} \frac{z_{2j-1}+z_{2j}}{\sqrt{2}}.\]
	It holds that $\|\mathbf{H}_\varphi\|=\|\varphi\|_{H^2(\mathbb{T}^\infty)}=1$, but for no $2<p\leq\infty$ is there an element $\psi$ in $L^p(\mathbb{T}^\infty)$ such that $P\psi = \varphi$. 
\end{theorem}

The first statement of Theorem~\ref{thm:cex} is a direct consequence of Theorem~\ref{thm:mnorm}. For the second statement, we argue by duality and borrow a simple estimate from \cite{BBHOCP19}. The fact that there are $\varphi$ in $H^2(\mathbb{T}^\infty)$ such that $\mathbf{H}_\varphi$ is bounded, but such that there is no $\psi$ in $L^p(\mathbb{T}^\infty)$ with $P\psi = \varphi$ when $2<p\leq\infty$ can also be deduced from the method in \cite{OCS12} and said estimate (Lemma~\ref{lem:plower} below). The main novelty of Theorem~\ref{thm:cex} is therefore that we provide an explicit example.

Let us now turn to our second class of symbols. Recall that a function $f$ in $L^2(\mathbb{T}^d)$ is called \emph{$m$-homogeneous} if its Fourier coefficients are supported on the frequencies $\alpha$ in $\mathbb{Z}^d$ which satisfy the equation $\alpha_1+\alpha_2+\cdots+\alpha_d=m$.

Let $H^2_m(\mathbb{T}^d)$ be the subspace of $H^2(\mathbb{T}^d)$ comprised of $m$-homogeneous functions. The search for $m$-homogeneous symbols generating minimal norm Hankel operators is facilitated by our second main result. The proof is rather easy, but we believe that the result may be of some independent interest in due to the prominence played by $m$-homogeneous expansions in function theory on polydiscs (see~\cite{CG86,Rudin69}).
\begin{theorem}\label{thm:mhom} 
	Suppose that $\varphi$ is in $H^2_m(\mathbb{T}^d)$. Let $\mathbf{H}_{\varphi,k}$ denote the restriction of the Hankel operator $\mathbf{H}_\varphi$ to $H^2_k(\mathbb{T}^d)$ and let $\mathbf{0}$ denote the zero operator. 
	\begin{enumerate}
		\item[(a)] If $0 \leq k \leq m$, then $\mathbf{H}_{\varphi,k}$ maps $H^2_k(\mathbb{T}^d)$ to $\overline{H^2_{m-k}}(\mathbb{T}^d)$. Moreover, $\mathbf{H}_\varphi$ enjoys the orthogonal decomposition
		\[\mathbf{H}_\varphi = \left(\bigoplus_{k=0}^m \mathbf{H}_{\varphi,k}\right)\oplus \mathbf{0}.\]
		\item[(b)] If $0\leq k \leq m$, then $\mathbf{H}_{\varphi,k}^\ast$ is unitarily equivalent to $\mathbf{H}_{\widetilde{\varphi},m-k}$ for $\widetilde{\varphi}(z)=\overline{\varphi(\overline{z})}$. 
	\end{enumerate}
\end{theorem}

Using Theorem~\ref{thm:mhom} we will find polynomial symbols generating minimal norm Hankel operators, but which cannot be obtained by the recipe discussed above. As a byproduct we also obtain the following improvement on the lower bound \eqref{eq:OCS}. 
\begin{theorem}\label{thm:Cdlower} 
	Let $C_d$ denote the optimal constant in \eqref{eq:Cd}. If $d$ is even, then
	\[C_d \geq \left(\frac{5\pi}{\pi+6\sqrt{3}}\right)^\frac{d}{2}.\]
\end{theorem}

The lower bound in Theorem~\ref{thm:Cdlower} can improved slightly by testing against a better function in the proof below. Conversely, the lower bound in \eqref{eq:OCS} is the best possible which can be obtained from the symbol $\varphi(z)=z_1+z_2$. As explained in \cite[Sec.~3]{Brevig19}, the optimal solution to the Nehari problem is in this case
\[\psi(z) = \sum_{k\in\mathbb{Z}} \frac{(-1)^k}{1-2k} z_1^{1-k} z_2^{k}\]
for $z$ on $\mathbb{T}^2$. It also follows from the arguments in \cite{Brevig19} that $\|\psi\|_{L^\infty(\mathbb{T}^2)}=\pi/2$.

The present paper is comprised of two additional sections. In Section~\ref{sec:inner} we establish Theorem~\ref{thm:inner1}, Theorem~\ref{thm:mnorm} and Theorem~\ref{thm:cex}. Section~\ref{sec:mhom} is devoted to the study of $m$-homogeneous symbols of Hankel operators and contains the proof of Theorem~\ref{thm:mhom} and Theorem~\ref{thm:Cdlower}.

\section{Symbols generated by inner functions} \label{sec:inner} 
In the proof of Theorem~\ref{thm:inner1} we will use the inner-outer factorization of functions in $H^p(\mathbb{T})$, for which our standard reference is Duren's monograph \cite[Ch.~2]{Duren70}. Every non-trivial function $f$ in $H^p(\mathbb{T})$ can be written as $f = I F$, where $I$ is inner and $F$ is outer. The factorization is unique up to a unimodular constant. In particular, it holds that $\|f\|_{H^p(\mathbb{T})}=\|F\|_{H^p(\mathbb{T})}$ and $F$ may be represented as
\begin{equation}\label{eq:outerrep} 
	F(z) = \exp\left(\int_0^{2\pi} \frac{e^{i\theta}+z}{e^{i\theta}-z}\log|f(e^{i\theta})|\,\frac{d\theta}{2\pi} \right). 
\end{equation}
We stress that $F$ does not vanish in the unit disc $\mathbb{D}$, so $\sqrt{F}$ will be analytic in $\mathbb{D}$. 
\begin{proof}[Proof of Theorem~\ref{thm:inner1}] 
	We explained in the introduction that if $\varphi = C I$ for a constant $C$ and an inner function $I$, then $\mathbf{H}_\varphi$ is easily seen to have minimal norm by \eqref{eq:lu}. Our job is therefore to establish the converse statement. Suppose therefore that $\varphi$ is a non-trivial element in $H^2(\mathbb{T})$ and that $\mathbf{H}_\varphi$ has minimal norm. Factor 
	\[\varphi = I \Phi,\]
	where $I$ is inner and $\Phi$ is outer, so that $\|\Phi\|_{H^2(\mathbb{T})}=\|\varphi\|_{H^2(\mathbb{T})}.$ Then $f = I \sqrt{\Phi}$ and $g = \sqrt{\Phi}$ satisfy
	\[\|f\|_{H^2(\mathbb{T})}=\|g\|_{H^2(\mathbb{T})}=\|\Phi\|_{H^1(\mathbb{T})}^{1/2}.\]
	By our assumption that $\mathbf{H}_\varphi$ has minimal norm, we find that
	\[\|\Phi\|_{H^2(\mathbb{T})}=\|\varphi\|_{H^2(\mathbb{T})} = \|\mathbf{H}_\varphi\| \geq \frac{\left|\left\langle \overline{P}(\overline{\varphi} f), \overline{g}\right\rangle\right|}{\|f\|_{H^2(\mathbb{T})}\|g\|_{H^2(\mathbb{T})}} = \frac{\|\Phi\|_{H^2(\mathbb{T})}^2}{\|\Phi\|_{H^1(\mathbb{T})}},\]
	where we used that $\overline{P}$ is self-adjoint and $\overline{P}(\overline{g})=\overline{g}$ in the final equality. This shows that $\|\Phi\|_{H^2(\mathbb{T})} \leq \|\Phi\|_{H^1(\mathbb{T})}$, which by the Cauchy--Schwarz inequality implies that there is some constant $C>0$ such that $|\Phi(e^{i\theta})|=C$ for almost every $e^{i\theta} \in \mathbb{T}$. By the representation \eqref{eq:outerrep} we conclude that $\varphi = C I$ (up to a unimodular constant). 
\end{proof}
\begin{remark}
	The proof of Theorem~\ref{thm:inner1} presented above is inspired by the modern proof of Nehari's theorem attributed to Helson (see \cite{Sarason11}). This argument exploits the inner-outer factorization to demonstrate that if $\mathbf{H}_\varphi$ is bounded, then $\varphi$ defines a bounded linear functional on $H^1(\mathbb{T})$ with $\|\varphi\|_{(H^1(\mathbb{T}))^\ast}=\|\mathbf{H}_\varphi\|$. The Hahn--Banach Theorem and the Riesz Representation Theorem can now be combined to show that there is some $\psi$ in $L^\infty(\mathbb{T})$ with $P\psi = \varphi$ and $\|\psi\|_{L^\infty(\mathbb{T})}=\|\varphi\|_{(H^1(\mathbb{T}))^\ast}$. 
\end{remark}

We require two preliminary results for the proof of Theorem~\ref{thm:mnorm}. The first is a special case of \cite[Lem.~2]{BP15}, which contains the corresponding result for all Schatten norms. A simpler proof of the present special case can be found in \cite[Lem.~4.4]{Sovik17}. 
\begin{lemma}\label{lem:hankelprod} 
	Suppose that $\varphi_1$ and $\varphi_2$ in $H^2(\mathbb{T}^d)$ depend on separate variables. If both $\mathbf{H}_{\varphi_1}$ and $\mathbf{H}_{\varphi_2}$ are bounded, then $\|\mathbf{H}_{\varphi_1\varphi_2}\| = \|\mathbf{H}_{\varphi_1}\| \|\mathbf{H}_{\varphi_2}\|$. 
\end{lemma}
\begin{lemma}\label{lem:hankelsum} 
	Suppose that $\varphi_1$ and $\varphi_2$ in $H^2(\mathbb{T}^d)$ depend on separate variables and that $\varphi_1(0)=\varphi_2(0)=0$. If both $\mathbf{H}_{\varphi_1}$ and $\mathbf{H}_{\varphi_2}$ are bounded, then
	\[\|\mathbf{H}_{\varphi_1+\varphi_2}\|^2 \leq \|\mathbf{H}_{\varphi_1}\|^2 + \|\mathbf{H}_{\varphi_2}\|^2.\]
\end{lemma}
\begin{proof}
	To avoid trivialities, we assume that $\|\varphi_j\|_{H^2(\mathbb{T}^d)}\neq0$ for $j=1,2$. Every $f$ in $H^2(\mathbb{T}^d)$ with $\|f\|_{H^2(\mathbb{T}^d)}=1$ can be orthogonally decomposed as
	\[f = \frac{t_1}{\|\varphi_1\|_{H^2(\mathbb{T}^d)}} \varphi_1 + \frac{t_2}{\|\varphi_2\|_{H^2(\mathbb{T}^d)}} \varphi_2 + t_3 g.\]
	Here $t_1,t_2,t_3$ are complex numbers satisfying $|t_1|^2+|t_2|^2+|t_3|^2=1$ and $g$ is a function which is orthogonal to both $\varphi_1$ and $\varphi_2$ and which satisfies $\|g\|_{H^2(\mathbb{T}^d)}=1$. A direct computation based on \eqref{eq:hankel} shows that
	\[\mathbf{H}_{\varphi_j} f = t_j \|\varphi_j\|_{H^2(\mathbb{T}^d)} + t_3 \mathbf{H}_{\varphi_j} g,\]
	for $j=1,2$, since $\varphi_1$ and $\varphi_2$ depend on separate variables and $\varphi_1(0)=\varphi_2(0)=0$. We now have the orthogonal decomposition
	\[\mathbf{H}_{\varphi_1+\varphi_2} f = \big(t_1 \|\varphi_1\|_{H^2(\mathbb{T}^d)}+t_2\|\varphi_2\|_{H^2(\mathbb{T}^d)}\big)+t_3\mathbf{H}_{\varphi_1}g + t_3\mathbf{H}_{\varphi_2} g.\]
	By orthogonality and the fact that $\|g\|_{H^2(\mathbb{T}^d)}=1$, we get
	\[\|\mathbf{H}_{\varphi_1+\varphi_2} f\|_{H^2(\mathbb{T}^d)}^2 \leq \big|t_1 \|\varphi_1\|_{H^2(\mathbb{T}^d)}+t_2\|\varphi_2\|_{H^2(\mathbb{T}^d)}\big|^2 + |t_3|^2 \big(\|\mathbf{H}_{\varphi_1}\|^2 + \|\mathbf{H}_{\varphi_2}\|^2\big).\]
	Using the Cauchy--Schwarz inequality on the first term and exploiting the general lower bound $\|\varphi_j\|_{H^2(\mathbb{T}^d)}\leq\|\mathbf{H}_{\varphi_j}\|$ from \eqref{eq:lu} for $j=1,2$, we get
	\[\|\mathbf{H}_{\varphi_1+\varphi_2} f\|_{H^2(\mathbb{T}^d)}^2 \leq (|t_1|^2+|t_2|^2+|t_3|^2)\big(\|\mathbf{H}_{\varphi_1}\|^2 + \|\mathbf{H}_{\varphi_2}\|^2\big) = \|\mathbf{H}_{\varphi_1}\|^2 + \|\mathbf{H}_{\varphi_2}\|^2.\]
	This completes the proof since $f$ is an arbitrary norm $1$ element in $H^2(\mathbb{T}^d)$. 
\end{proof}

\begin{proof}
	[Proof of Theorem~\ref{thm:mnorm}] We begin with (a), where Lemma~\ref{lem:hankelprod} and the assumption that $\mathbf{H}_{\varphi_1}$ and $\mathbf{H}_{\varphi_2}$ have minimal norm imply that
	\[\|\mathbf{H}_{\varphi_1\varphi_2}\| = \|\mathbf{H}_{\varphi_1}\| \|\mathbf{H}_{\varphi_2}\| = \|\varphi_1\|_{H^2(\mathbb{T}^d)} \|\varphi_2\|_{H^2(\mathbb{T}^d)} = \|\varphi_1 \varphi_2\|_{H^2(\mathbb{T}^d)}.\]
	The final equality is a trivial consequence of the fact that $\varphi_1$ and $\varphi_2$ depend on separate variables. Hence $\mathbf{H}_{\varphi_1\varphi_2}$ is has minimal norm. In the case (b), we similarly get from Lemma~\ref{lem:hankelsum} and the assumption that $\mathbf{H}_{\varphi_1}$ and $\mathbf{H}_{\varphi_2}$ have minimal norm that
	\[\|\mathbf{H}_{\varphi_1+\varphi_2}\|^2 \leq \|\mathbf{H}_{\varphi_1}\|^2 + \|\mathbf{H}_{\varphi_2}\|^2 = \|\varphi_1\|_{H^2(\mathbb{T}^d)}^2 +\|\varphi_2\|_{H^2(\mathbb{T}^d)}^2 = \|\varphi_1+\varphi_2\|_{H^2(\mathbb{T}^d)}^2.\]
	The final equality holds because $\varphi_1 \perp \varphi_2$. Hence $\mathbf{H}_{\varphi_1+\varphi_2}$ has minimal norm. 
\end{proof}

We require following estimate in the proof of the second part of Theorem~\ref{thm:cex}. 
\begin{lemma}\label{lem:plower} 
	Suppose that $1 \leq q \leq 2$. Then
	\[\left\|\frac{z_1+z_2}{\sqrt{2}}\right\|_{H^q(\mathbb{T}^2)}^{-1} \geq 1 + \frac{2\log{2}-1}{8}(2-q).\]
\end{lemma}
\begin{proof}
	We first extract from the proof of \cite[Lem.~21]{BBHOCP19} the estimate
	\[\|\varphi\|_{H^q(\mathbb{T}^2)}^{-1} \geq \frac{1}{\sqrt{2}} \left(1+\frac{q}{2}\right)^{\frac{1}{q}}.\]
	The proof is completed by using Taylor's theorem at $q=2$. 
\end{proof}
\begin{proof}
	[Proof of Theorem~\ref{thm:cex}] For every positive integer $k$, let
	\[\varphi_k(z) = \prod_{j=(k-1)k/2+1}^{k(k+1)/2} \frac{z_{2j-1}+z_{2j}}{\sqrt{2}}\]
	and note that $\|\varphi_k\|_{H^2(\mathbb{T}^\infty)}=1$. By the recipe outlined after Theorem~\ref{thm:mnorm}, it is clear that $\|\mathbf{H}_\varphi\|=\|\varphi\|_{H^2(\mathbb{T}^\infty)}=1$ if
	\[\varphi(z) = \frac{\sqrt{6}}{\pi} \sum_{k=1}^\infty \frac{\varphi_k(z)}{k}.\]
	It remains to establish the second claim, where we shall argue by contradiction. Fix $2<p\leq\infty$ and assume that there is some $\psi$ in $L^p(\mathbb{T}^\infty)$ such that $P\psi = \varphi$. Since $P$ is self-adjoint, we get from H\"older's inequality that 
	\begin{equation}\label{eq:contradictme} 
		\frac{|\langle f,\varphi \rangle |}{\|f\|_{H^q(\mathbb{T}^\infty)}} \leq \|\psi\|_{L^p(\mathbb{T}^\infty)}<\infty 
	\end{equation}
	for every non-trivial function $f$ in $H^q(\mathbb{T}^\infty)$, where $q = p/(p-1)$. The fact that $2<p\leq\infty$ means that $1 \leq q <2$. Choosing $f = \varphi_k$, we see that
	\[\lim_{k \to \infty} \frac{\langle \varphi_k, \varphi \rangle}{\|\varphi_k\|_{H^q(\mathbb{T}^\infty)}} = \frac{\sqrt{6}}{\pi}\lim_{k \to \infty} \frac{1}{k} \frac{\|\varphi_k\|_{H^2(\mathbb{T}^\infty)}^2}{\|\varphi_k\|_{H^q(\mathbb{T}^\infty)}} = \frac{\sqrt{6}}{\pi} \lim_{k \to \infty} \frac{1}{k}\frac{1}{\|\varphi_1\|_{H^q(\mathbb{T}^2)}^k} = \infty,\]
	by Lemma~\ref{lem:plower}. This contradicts \eqref{eq:contradictme} and hence our assumption that there is some $\psi$ in $L^q(\mathbb{T}^\infty)$ with $P\psi = \varphi$ must be wrong. 
\end{proof}

\section{\texorpdfstring{$m$}{m}-homogeneous symbols} \label{sec:mhom} 
To prepare for the proof of Theorem~\ref{thm:mhom}, we first orthogonally decompose $H^2(\mathbb{T}^d)$ and $\overline{H^2}(\mathbb{T}^d)$ using $m$-homogeneous functions. It is clear that 
\begin{equation}\label{eq:decomp} 
	H^2(\mathbb{T}^d) = \bigoplus_{m=0}^\infty H^2_m(\mathbb{T}^d) \qquad\text{and} \qquad \overline{H^2}(\mathbb{T}^d) = \bigoplus_{m=0}^\infty \overline{H^2_m}(\mathbb{T}^d). 
\end{equation}
Note that the functions in $\overline{H^2_m}(\mathbb{T}^d)$ are $-m$-homogeneous, since they are precisely the complex conjugates of functions from $H^2_m(\mathbb{T}^d)$. 

\begin{proof}[Proof of Theorem~\ref{thm:mhom}] 
	To establish (a), decompose a function $f$ in $H^2(\mathbb{T}^d)$ as
	\[f = \sum_{k=0}^\infty f_k\]
	in view of \eqref{eq:decomp}. By assumption, our symbol $\varphi$ is $m$-homogeneous. Hence we have
	\[\overline{\varphi} f = \sum_{k=0}^\infty \overline{\varphi} f_k,\]
	and $\overline{\varphi} f_k$ is $(-m+k)$-homogeneous. Since homogenity is preserved under $\overline{P}$, this shows that $\mathbf{H}_\varphi$ maps $H_k^2(\mathbb{T}^d)$ to $\overline{H^2_{m-k}}(\mathbb{T}^d)$ when $0 \leq k \leq m$. This completes the proof of the first claim. If $k \geq m+1$, then $\overline{\varphi}f_k$ has positive homogeneity and hence $\overline{P}(\overline{\varphi} f_k)=0$. Combining what we have done with \eqref{eq:decomp} shows that $\mathbf{H}_\varphi$ enjoys the stated orthogonal decomposition 
	\begin{equation}\label{eq:adecomp} 
		\mathbf{H}_\varphi = \left(\bigoplus_{k=0}^m \mathbf{H}_{\varphi,k}\right)\oplus \mathbf{0}. 
	\end{equation}
	For the proof of (b), we first check that if $f$ is in $H^2(\mathbb{T}^d)$ and is $g$ in $\overline{H^2}(\mathbb{T}^d)$, then
	\[\langle \mathbf{H}_\varphi f, g \rangle = \langle \overline{P}(\overline{\varphi}f), g \rangle = \langle f, \varphi g \rangle = \langle f, P(\varphi g) \rangle,\]
	which shows that $\mathbf{H}_\varphi^\ast g = P(\varphi g)$. This also shows that $\mathbf{H}_\varphi^\ast$ is unitarily equivalent to the Hankel operator $\mathbf{H}_{\widetilde{\varphi}}$ where $\widetilde{\varphi}(z)=\overline{\varphi(\overline{z})}$. The decomposition \eqref{eq:adecomp} therefore applies to $\mathbf{H}_{\widetilde{\varphi}}$, which means that $\mathbf{H}_{\varphi,k}^\ast$ must be unitarily equivalent to $\mathbf{H}_{\widetilde{\varphi},k-m}$. 
\end{proof}

The following result illustrates how Theorem~\ref{thm:mhom} pertains to minimal norm Hankel operators. Part (a) allows us to focus on the restricted Hankel operators and part~(b) reduces the number of restricted Hankel operators we need to consider.
\begin{corollary}\label{cor:decomplower} 
	Let $\varphi$ be in $H^2_m(\mathbb{T}^d)$. Then $\mathbf{H}_\varphi$ has minimal norm if and only if 
	\begin{equation}\label{eq:decomplower} 
		\max_{0<k\leq \lfloor m/2 \rfloor} \|\mathbf{H}_{\varphi,k} \| \leq \|\varphi\|_{H^2(\mathbb{T}^d)} 
	\end{equation}
	where $\mathbf{H}_{\varphi,k}$ denotes the restriction of $\mathbf{H}_\varphi$ to $H^2_k(\mathbb{T}^d)$. 
\end{corollary}
\begin{proof}
	It is clear from Theorem~\ref{thm:mhom} (a) that $\mathbf{H}_\varphi$ has minimal norm if and only if 
	\begin{equation}\label{eq:decreaseme} 
		\max_{0 \leq k \leq m} \|\mathbf{H}_{\varphi,k} \| \leq \|\varphi\|_{H^2(\mathbb{T}^d)}, 
	\end{equation}
	so our goal is to demonstrate that the set we take the maxima over may be decreased to obtain \eqref{eq:decomplower}. From Theorem~\ref{thm:mhom} (b) we find that
	\[\|\mathbf{H}_{\varphi,k}\| = \|\mathbf{H}_{\varphi,k}^\ast\|=\|\mathbf{H}_{\widetilde{\varphi},m-k}\|=\|\mathbf{H}_{\varphi,m-k}\|\]
	for $0 \leq k \leq m$, which allows us to decrease the set in \eqref{eq:decreaseme} to $0 \leq k \leq \lfloor m/2 \rfloor$. It remains to exclude the case $k=0$. Since $H^2_0(\mathbb{T}^d)$ is comprised of constant functions, it follows at once from the definition of $\mathbf{H}_\varphi$ that $\|\mathbf{H}_{\varphi,0}\|=\|\varphi\|_{H^2(\mathbb{T}^d)}$. Hence the desired inequality is automatically satisfied for $k=0$. 
\end{proof}
\begin{remark}
	Since the maximum in \eqref{eq:decomplower} is taken over an empty set of integers if $m=1$, Corollary~\ref{cor:decomplower} ensures that 
	\begin{equation}\label{eq:1hom} 
		\varphi_1(z) = \sum_{j=1}^d c_j z_j 
	\end{equation}
	generates a minimal norm Hankel operator for any choice of coefficients. This can also be seen from the recipe inspired by Theorem~\ref{thm:mnorm}. The $1$-homogeneous symbols \eqref{eq:1hom} are used in \cite{BP15} to extend the result of \cite{OCS12} to certain Schatten classes. 
\end{remark}

We can put Corollary~\ref{cor:decomplower} to use and easily obtain the following concrete examples. It is clear that if $a$ and $b$ are positive, then we cannot construct the polynomials $\varphi_2$ and $\varphi_3$ using the recipe inspired by Theorem~\ref{thm:mnorm}. 
\begin{theorem}\label{thm:twoguys} 
	Consider the polynomials
	\[\varphi_2(z) = z_1^2+ a z_1 z_2 + z_2^2 \qquad \text{and}\qquad \varphi_3(z) = z_1^3 + b z_1^2 z_2 + b z_1 z_2^2 + z_2^3\]
	where $a$ and $b$ are nonnegative real numbers. The Hankel operator 
	\begin{enumerate}
		\item[(a)] $\mathbf{H}_{\varphi_2}$ has minimal norm if and only if $a\leq 1/2$, 
		\item[(b)] $\mathbf{H}_{\varphi_3}$ has minimal norm if and only if $b \leq \sqrt{2}-1$. 
	\end{enumerate}
\end{theorem}
\begin{proof}
	[Proof of Theorem~\ref{thm:twoguys} (a)] The function $\varphi_2$ is $2$-homogeneous. By Corollary~\ref{cor:decomplower}, it is sufficient to check which coefficients $a$ ensure that the inequality 
	\begin{equation}\label{eq:satisfya} 
		\|\mathbf{H}_{\varphi_2,1}\| \leq \sqrt{2+a^2} 
	\end{equation}
	is satisfied. The matrix representation of the operator $\mathbf{H}_{\varphi_2,1}\colon H^2_1(\mathbb{T}^2) \to \overline{H^2_1}(\mathbb{T}^2)$ with respect to the standard basis is
	\[M_{\varphi_2,1} = 
	\begin{pmatrix}
		1 & a \\
		a & 1 
	\end{pmatrix}
	.\]
	The norm of this matrix is seen to be $1+a$, since $a\geq0$ by assumption. The requirement \eqref{eq:satisfya} becomes $1+a \leq \sqrt{2+a^2}$, which simplifies to $a\leq 1/2$. 
\end{proof}
\begin{proof}
	[Proof of Theorem~\ref{thm:twoguys} (b)] The function $\varphi_3$ is $3$-homogeneous. By Corollary~\ref{cor:decomplower}, it is sufficient to check which nonnegative coefficients $b$ ensure that 
	\begin{equation}\label{eq:satisfyb} 
		\|\mathbf{H}_{\varphi_3,1}\| \leq \sqrt{2+2b^2}. 
	\end{equation}
	The matrix representation of the operator $\mathbf{H}_{\varphi_3,1}\colon H^2_1(\mathbb{T}^2) \to \overline{H^2_2}(\mathbb{T}^2)$ with respect to the standard basis is
	\[M_{\varphi_3,1} = 
	\begin{pmatrix}
		1 & b \\
		b & b \\
		b & 1 
	\end{pmatrix}
	\qquad\text{and hence} \qquad M_{\varphi_3,1}^\ast M_{\varphi_3,1} = 
	\begin{pmatrix}
		1+2b^2 & 2b+b^2 \\
		2b+b^2 & 1+2b^2 
	\end{pmatrix}
	.\]
	Since $b\geq0$ it is easy to see that the norm of the latter matrix is $1+2b+3b^2$. The requirement \eqref{eq:satisfyb} becomes $1+2b+3b^2 \leq 2+2b^2$, which simplifies to $b \leq \sqrt{2}-1$. 
\end{proof}
\begin{proof}
	[Proof of Theorem~\ref{thm:Cdlower}] A simple argument based on Lemma~\ref{lem:hankelprod} shows that if $d$ is an even integer, then $C_d \geq C_2^{d/2}$. It is therefore sufficient to establish that 
	\begin{equation}\label{eq:C2} 
		C_2 \geq \frac{5\pi}{\pi+6\sqrt{3}}. 
	\end{equation}
	Starting from the definition of $C_2$ from \eqref{eq:Cd} and arguing as in the proof of the second part of Theorem~\ref{thm:cex}, we get that 
	\begin{equation}\label{eq:duallower} 
		C_2 \geq \frac{|\langle f, \varphi \rangle|}{\|\mathbf{H}_\varphi\| \, \|f\|_{H^1(\mathbb{T}^2)}} 
	\end{equation}
	for any pair of non-trivial functions $f$ in $H^1(\mathbb{T}^2)$ and $\varphi$ in $H^2(\mathbb{T}^d)$. We will choose
	\[f(z)= z_1^2 + z_1z_2+z_2^2 \qquad \text{and} \qquad \varphi(z) = z_1^2+\frac{z_1 z_2}{2}+z_2^2.\]
	Clearly $\langle f, \varphi \rangle = 5/2$ and by Theorem~\ref{thm:twoguys} (a) we know that $\|\mathbf{H}_\varphi\|=\|\varphi\|_{H^2(\mathbb{T}^2)}=3/2$. Since the coefficients of $f$ are real and since $f$ is $2$-homogeneous, we can simplify
	\[\|f\|_{H^1(\mathbb{T}^2)} = \int_0^{2\pi}\int_0^{2\pi} \big|f(e^{i\theta_1},e^{i\theta_2})\big|\,\frac{d\theta_1}{2\pi}\frac{d\theta_2}{2\pi} = \int_0^\pi \big|e^{i\theta}+1+e^{-i\theta}\big|\,\frac{d\theta}{\pi}.\]
	Using that $e^{i\theta}+e^{-i\theta}=2\cos{\theta}$ and that the solution to the equation $2\cos{\theta}+1=0$ on the interval $0\leq \theta \leq \pi$ is $\theta = 2\pi/3$, we find that
	\[\|f\|_{H^1(\mathbb{T}^2)} = \int_0^{2\pi/3} (2\cos{\theta}+1)\,\frac{d\theta}{\pi}-\int_{2\pi/3}^\pi (2\cos{\theta}+1)\,\frac{d\theta}{\pi} = \frac{2}{3}+\frac{\sqrt{3}}{\pi}-\left(\frac{1}{3}-\frac{\sqrt{3}}{\pi}\right).\]
	Inserting everything into \eqref{eq:duallower} and tidying up yields the stated lower bound \eqref{eq:C2}. 
\end{proof}
\begin{remark}
	Some cursory numerical experiments indicate that it might be optimal to choose $a=1/2$ in Theorem~\ref{thm:twoguys} (a). For this choice of symbol $\varphi$, it is optimal to choose $f(z)=z_1^2+c z_1z_2 + z_2^2$ for some $0.8<c<0.9$.
\end{remark}

\bibliographystyle{amsplain} 
\bibliography{minimal}

\end{document}